\documentclass[10pt]{article}
\usepackage{amsmath, amssymb, amsthm}
\usepackage{graphicx}
\usepackage{enumerate}
\usepackage{epstopdf}
\begin{document}
\title{\bf Existence of asymptotic almost automorphic solutions of fractional integro differential equations with non local conditions}
\author{Syed Abbas$^1$\thanks{
Tel.\ +91-190-523-7933, Fax.\ +91-190-523-7942, \
E-mail: sabbas.iitk@gmail.com, \
abbas@iitmandi.ac.in},
Gaston M. N'Guerekata$^2$
\\\\
$^1$School of Basic Sciences \\
Indian Institute of Technology Mandi \\
Mandi, H.P., 175001, India. \\\\
$^2$ Department of Mathematics \\
Morgan State University, 1700 E. Cold Spring Lane \\ Baltimore, MD 21251, USA.
}

\begin{titlepage}

\maketitle

\begin{abstract}
This paper deals with the existence of asymptotic almost automorphic solution of fractional integro differential equation. We prove the result by using fixed point theorems. We show the result with Lipschitz condition and without Lipschitz condition on the forcing term. At the end examples have been given to illustrate the analytical findings.
\end{abstract}


%\begin{keyword}
%Caputo fractional derivative; Stability, Fixed point method, fractional Adams-Bashforth-Moulton Method.
%\end{keyword}

\end{titlepage}

\newtheorem{theorem}{Theorem}[section]
\newtheorem{lemma}[theorem]{Lemma}
\newtheorem{proposition}[theorem]{Proposition}
\newtheorem{remark}[theorem]{Remark}
\newtheorem{definition}[theorem]{Definition}
\newtheorem{example}[theorem]{Example}

\section{Introduction}
\label{s:1}

There are two major kind of fractional derivatives \cite{diet1}, one defined in the sense of Caputo and another in Riemann-Liouville. Both are defined as follows
$$^{C}D_t^{\alpha} f(t)=\frac{1}{\Gamma (n-\alpha)}\int_0^t(t-s)^{n-{\alpha}-1}f^{(n)}(s)ds,$$ $$\ D_t^{\alpha} f(t)=\frac{1}{\Gamma (n-\alpha)} \frac{d^n}{dt^n}\int_0^t(t-s)^{n-{\alpha}-1}f(s)ds,$$ where the symbol $C$ is meant for Caputo's derivative. The above integrals exists
provided the right hand sides exist point-wise on $[0,T]$ and $\Gamma$ denotes the gamma function. Using the \emph{Riemann-Liouville fractional integral} \cite{diet1}
$I^{\alpha}_0 f(t)=\frac{1}{\Gamma(\alpha)}\int_0^t(t-s)^{\alpha-1}f(s)ds$, we have
$$^{C}D_t^{\alpha} f(t) =I^{n-\alpha}_0\frac{d^n}{dt^n}f(t) \ \mbox{and} \ D_t^{\alpha} f(t)=\frac{d^n}{dt^n}I^{n-\alpha}_0f(t)$$. $I^{\alpha}_0f$ exists, for instance, for all $\alpha >0,$ if $f\in C^0([0,T])\cap L^1_{loc}([0,T])$, moreover $I^{\alpha}_0f(0) = 0$.

One can see that the both the fractional derivatives are actually nonlocal operator because integral is a nonlocal operator.
Moreover, calculating time fractional derivative of a function at some
time requires all the past history and hence fractional derivatives can be used for modeling systems with
memory. Fractional differential equations can be formulated using both Caputo or Riemann-Liouville fractional derivatives.

Asymptotic almost automorphic functions has been introduced in the literature by N'Guerekata \cite{ngur}. There has been many application of this function in the theory of differential equations, for more details one could see the \cite{diagana,ding,liang,lizama, zhao} and references therein. As one could see from last few decades fractional differential equations got lot of attention from mathematicians, scientists, engineers because of their application in various fields like viscoelasticity, anomalous diffusion etc. Recently, many authors have established results like almost automorphic, asymptotic almost automorphic, pseudo almost automorphic \cite{agarwal, araya, abbas, abbas1, meril, mophou1, mophou2,zhao}.

Zhao et. al. established asymptotically almost automorphic solution of the following fractional differential equation,
\begin{eqnarray}
^{C}D_t^{\alpha}u(t) &=&Au(t)+{^{C}D}_t^{\alpha-1}f(t,u(t)),\nonumber \\
&& u(0)=u_0, \quad t \in \mathbb{R}^+, \ 1< \alpha <2,
\label{EE1}
\end{eqnarray}
where $A$ is linear densely defined sectorial operator.

Motivated by above work, in this work we shall study the existence of asymptotic almost automorphic  solutions of the following fractional differential equations of Caputo's type of order $\alpha \in (1,2),$
\begin{eqnarray}
^{C}D_t^{\alpha}u(t) &=&Au(t)+{^{C}D}_t^{\alpha-1}f(t,u(t),Ku(t)),\nonumber \\
Ku(t) &=& \int_{-\infty}^tk(t-s)u(s)ds, \nonumber \\
u(0)&+&g(u)=u_0, \quad t \in \mathbb{R},
\label{EE2}
\end{eqnarray}
where $A:D(A) \subset X \rightarrow X$ is a linear densely defined sectorial operator and $X$ is a complex Banach space. The initial point $u_0 \in X$ and $f \in AAA(\mathbb{R} \times X \times X, X), k \in L^1(\mathbb{R}^+).$
The organization of the paper is as follows: In section-2, we give some basic definitions and results. In section-2,  we establish the existence of asymptotic almost automorpchic solution of the equation (\ref{EE2}). At end in section-2, we give examples to illustrate the analytical findings.

\section{Preliminaries}

We assume that the spaces $X, Y$ equipped with the norms $\|\cdot\|_X$ and $\|\cdot\|_Y$ are Banach spaces respectively. The notation $\mathcal{C}(\mathbb{R}, X)$ denote the collection of all continuous functions from $\mathbb{R}$ to $X.$ The space of bounded continuous functions from $\mathbb{R}$ to $X$ is denoted by $\mathcal{BC}(\mathbb{R}, X).$ The space $\mathcal{BC}(\mathbb{R}, X)$ equipped with the norm $\|\phi\|_{\infty}=\sup_{t \in \mathbb{R}}\|\phi(t)\|$ is a Banach space. Moreover the notion $\mathbf{B}(X, Y)$ denote the space of bounded linear operators from $X$ to $Y.$ We write $\mathbf{B}(X)$ when $X=Y.$

The definition of almost automorphic operator has been given by
N'Gu$\acute{e}$r$\acute{e}$kata and Pankov \cite{guerekata}. Now
we state the definitions, for more details one can see \cite{ngur}.
\begin{definition}\label{aa} \rm
A bounded continuous function $f:\mathbb{R}\to X$ is called
almost automorphic if for every real sequence $(s_{n})$, there
exists a subsequence $(s_{n_k})$ such that
$$
g(t)=\lim_{n\to\infty}f(t+s_{n_k})
$$
is well defined for each $t\in \mathbb{R}$ and
$$
\lim_{n\to\infty}g(t-s_{n_k})=f(t)
$$
for each $t\in \mathbb{R}$. Denote by $AA(\mathbb{R}, X)$ the set of
all such functions.
\end{definition}

\begin{definition} \rm
A bounded continuous function $f:\mathbb{R}\times X \to X$ is
called almost automorphic in $t$ uniformly for $x$ in compact
subsets of $X$ if for every compact subset $K$ of $X$ and every
real sequence $(s_{n})$, there exists a subsequence $(s_{n_k})$
such that
$$g(t,x)=\lim_{n\to\infty}f(t+s_{n_k},x)$$
is well defined for each $t\in \mathbb{R}$, $x\in K$ and
$$\lim_{n\to\infty}g(t-s_{n_k},x)=f(t,x)$$
for each $t\in \mathbb{R}$, $x\in K$. Denote by
$AA(\mathbb{R}\times X, X)$ the set of all such functions.
\end{definition}
\noindent
The space of all bounded continuous functions $h:\mathbb{R}\rightarrow X$ such that 
$\|h(t)\| \rightarrow 0$ as $t\rightarrow \infty$ is denoted by $\mathcal{C}_0(\mathbb{R}, X).$ Moreover, we denote $\mathcal{C}_0(\mathbb{R}\times X, X)$ the space of all bounded continuous functions from
 $\mathbb{R} \times X$ to $X$ such that $\lim_{t\rightarrow \infty} \|h(t,x)\|=0$ in $t$ and uniformly on any bounded subset of $X.$

 \begin{definition}\label{aa} \rm
A bounded continuous function $f \in \mathcal{C}_0(\mathbb{R}, X)$ is called
asymptotically almost automorphic iff it can be written as $f=f_1+f_2,$ where $f_1 \in AA(\mathbb{R}, X)$ and $f_2 \in \mathcal{C}_0(\mathbb{R}, X).$ This kind of functions is denoted by $AAA(\mathbb{R}^+, X).$
\end{definition}

 \begin{definition}\label{aa} \rm
A bounded continuous function $I \in \mathcal{C}_0(\mathbb{R} \times, X)$ is called
asymptotically almost automorphic iff it can be written as $I=I_1+I_2,$ where $I_1 \in AA(\mathbb{R} \times X, X)$ and $I_2 \in \mathcal{C}_0(\mathbb{R} \times X, X).$ This kind of functions is denoted by $AAA(\mathbb{R} \times X, X).$
\end{definition}

We state a lemma by Liang et.al. \cite{liang} about the composition result.
\begin{lemma}
Let $f(t,x)=g(t,x)+\phi(t,x)$ is an asymptotically almost automorphic function with $g(t,x) \in AA(\mathbb{R}\times X, X)$ and $\phi(t,x) \in \mathcal{C}_0(\mathbb{R}^+ \times X, X)$ and $f(t,x)$ is uniformly continuous on any bounded subset $\Omega \subset X$
uniformly in $t.$ Then for $u(\cdot) \in AAA(\mathbb{R}, X),$ the function $f(\cdot, u(\cdot)) \in AAA(\mathbb{R} \times X, X).$
\end{lemma}

 A closed and linear operator $A$ is said to be sectorial of type
$\omega$ and angle $\theta$ if there exists $0 <\theta <
\frac{\pi}{2}, M> 0$ and $\omega \in \mathbb{R}$ such that its
resolvent exists outside the sector $$\omega+S_{\theta}:=
\{\omega+\lambda: \lambda \in \mathbb{C}, \ |arg(-\lambda)| <
\theta\},$$ and $$\|(\lambda- A)^{-1}\| \le
\frac{M}{|\lambda-\omega|}, \ \lambda \not \in \omega+
S_{\theta}.$$ Sectorial operators are well studied in the
literature. For a recent reference including several examples and
properties we refer the reader to \cite{haase}. Note that an
operator A is sectorial of type $\omega$ if and only if $\lambda I
- A$ is sectorial of type $0.$

The equation (\ref{EE2}) can be thought of a limiting case of the
following equation
             \begin{eqnarray}
v^{\prime}(t)=\int_0^t
\frac{(t-s)^{\alpha-2}}{\Gamma(\alpha-1)}Av(s)ds+f(t,u(t),Ku(t)), \
t \ge 0,
\label{meq1}
             \end{eqnarray}
in the sense that the solutions are asymptotic to each other as $t
\rightarrow \infty.$ If we consider that the operator $A$ is
sectorial of type $\omega$ with $\theta \in
[0,\pi(1-\frac{\alpha}{2})),$ then problem \ref{EE2} is well
posed \cite{custa}. Thus we can use variation of parameter
formulae to get
$$v(t)=S_{\alpha}(t)(u_0-g(u))+\int_0^tS_{\alpha}(t-s)f(s,u(s),Ku(s))ds, \quad t \ge 0,$$
where
$$S_{\alpha}(t)=\frac{1}{2\pi i}\int_{\gamma} e^{\lambda t} \lambda^{\alpha-1}(\lambda^{\alpha}I-A)^{-1}d\lambda,
\quad t \ge 0,$$ where the path $\gamma$ lies outside the sector
$\omega+S_{\theta}.$ If $S_{\alpha}(t)$ is integrable then the
solution of problem (\ref{EE2}) is given by
$$u(t)=\int_{-\infty}^t S_{\alpha}(t-s)f(s,u(s),Ku(s))ds.$$
Now one can easily see that
$$v(t)-u(t)=S_{\alpha}(t)(u_0-g(u))-\int_t^{\infty}S_{\alpha}(s)f(t-s,u(t-s),Ku(t-s))ds.$$
Hence for $f \in L^p({\mathbb{R}^+\times X \times X,X}), \ p \in
[1,\infty)$ we have $v(t)-u(t) \rightarrow 0$ as $t\rightarrow
\infty.$
             \begin{definition}
A function $u: \mathbb{R} \rightarrow X$ is said to be a mild
solution to \ref{EE2} if the function $S_{\alpha}(t-s)f(s,
u(s),Ku(s))$ is integrable on $(-\infty, t)$ for each $t \in
\mathbb{R}$ and $$u(t)= \int_{-\infty}^tS_{\alpha}(t-s)f(s,
u(s),Ku(s))ds,$$ for each $t \in \mathbb{R}.$
            \end{definition}
Recently, Cuesta in \cite{custa}, Theorem-$1$, has proved that if
$A$ is a sectorial operator of type $\omega< 0$ for some $M > 0$
and $\theta \in [0,\pi(1-\frac{\alpha}{2})),$ then there exists
$C> 0$ such that $$\|S_{\alpha}(t)\| \le
\frac{CM}{1+|\omega|t^{\alpha}}$$ for $t\ge0.$ Also the following
relation holds,
$$\int_0^{\infty}\frac{1}{1+|\omega|t^{\alpha}}=\frac{|\omega|^{\frac{-1}{\alpha}}\pi}{\alpha \sin
\frac{\pi}{\alpha}}$$ for $\alpha \in (1,2).$

Assume that the function $h \mathbb{R}^+ \rightarrow (1, \infty)$ is continuous and $h(t) \ge 1$ for all $t \in \mathbb{R}^+$ and $\lim_{t\rightarrow \infty }h(t)=0.$ Denote the space
$$\mathcal{C}_h(X)=\Big\{u\in \mathcal{C}(\mathbb{R}^+, X): \lim_{t\rightarrow \infty}\frac{u(t)}{h(t)}=0\Big\}.$$ The space $\mathcal{C}_h(X)$ is a Banach space equipped with the norm $$\|u\|_h=\sup_{t \in \mathbb{R}^+}\frac{\|u(t)\|}{h(t)}.$$

The following lemma is from \cite{agarwal00}.
\begin{lemma} \label{lemma3}
A subset $R \subset \mathcal{C}_h(X)$ is said to be relatively compact if
\begin{itemize}
\item[(a1)] The set $R_b=\{u|_{[0,b]}: u \in R \}$ is relatively compact in $\mathcal{C}([0,b],X)$ for all $b \ge 0,$
\item[(a2)] $\lim_{t\rightarrow \infty}\frac{\|u(t)\|}{h(t)}=0,$ uniformly for $u \in R.$
\end{itemize}
\end{lemma}

\begin{remark}
It is to note that the operator defined by $\int_0^t S_{\alpha}(t-s)f(s, u(s))$ may not be in $AAA(\mathbb{R}, X)$ when $f \in AAA(\mathbb{R}\times X, X).$

\end{remark}
\section{Asymptotically almost automorphic solution}
In order to prove our result, we need the following assumptions,
\begin{itemize}
\item[(H1)] The operator $A$ is a sectorial of type $\omega<0,$
\item[(H2)] The function $f \in AAA(\mathbb{R} \times X, X),$ there exists a positive constant $L_f$ such that
$$\|f(t,u,\phi)-f(t,v,\psi)\| \le L_f(\|u-v\|+\|\phi-\psi\|),$$ for all $(t, u, \phi), (t,v, \psi) \in \mathbb{R} \times X \times X.$
\item[(H3)] The operator $\{S_{\alpha}(t)\}_{t \ge 0}$ be a strongly continuous family of operators.
\end{itemize}
The following lemma is from \cite{ding}.
\begin{lemma}
Let $f=f_1+f_2 \in AAA(\mathbb{R} \times X, X)$ with $f_1 \in AA(\mathbb{R} \times X, X), f_2 \in \mathcal{C}_0(\mathbb{R} \times X, X)$ satisfying the hypothesis $(H2).$ If $u(t) \in AAA(\mathbb{R}, X),$ then $f(\cdot, u(\cdot)) \in AAA(\mathbb{R} \times X, X).$
\end{lemma}
The following lemma is from \cite{agarwal00}.
\begin{lemma}
Under the hypothesis $(H1),$ if the function $f \in AAA(\mathbb{R}, X)$ then the operator $Ff$ defined by
$$(Ff)(t)=\int_{-\infty}^t S_{\alpha}(t-s)f(s)ds, \quad t \in \mathbb{R}$$
is asymptotically almost automorphic.
\end{lemma}

We state the one of the main result of this paper.
\begin{theorem}
Under the hypothesis $(H1)-(H3),$ there exists a unique asymptotically almost automorphic mild solution of equation (\ref{EE2}) if
$$\Lambda=\frac{CM|\omega|^{-\frac{1}{\alpha}}\pi}{\alpha \sin(\frac{\pi}{\alpha})}L_f(1+\|k\|_{L^1(\mathbb{R}^+)}) <1.$$
\end{theorem}
\begin{proof}
Define the operator
$$Fu(t)=\int_{-\infty}^tS_{\alpha}(t-s)f(s,u(s),Ku(s))ds, \quad t \in \mathbb{R}.$$
The operator $F$ is well defined and continuous. As we know from the assumption $(H2)$ and composition theorem for asymptotically almost automorphic functions that $f \in AAA(\mathbb{R} \times X, X).$ Moreover, one can observe that
$$\int_{-\infty}^tS_{\alpha}(t-s)f(s,u(s),Ku(s))ds \in AAA(\mathbb{R}, X).$$
Now for $u, v \in AAA(\mathbb{R}\times X, X),$ we get
\begin{eqnarray}
&&\|Fu(t)-Fv(t)\| \nonumber \\ &&= \Big\| \int_{-\infty}^tS_{\alpha}(t-s)(f(s,u(s),Ku(s))-f(s,v(s),Kv(s)))ds \Big\|
\nonumber \\ && \le \int_0^{\infty}\|S_{\alpha}(s)\|_{\mathbf{B}(X)}\|f(t-s,u(t-s),Ku(t-s))\nonumber \\ &&-f(t-s,v(t-s),Kv(t-s))\|ds
\nonumber \\ && \le \int_0^{\infty} \frac{CM}{1+|\omega|s^{\alpha}} L_f(\|u(t-s)-v(t-s)\|+\|Ku(t-s)-Kv(t-s)\|) ds. \nonumber \\
\end{eqnarray}
Let us calculate
\begin{eqnarray}
\|Ku(t)-Kv(t)\| &&= \Big\| \int_{-\infty}^t k(t-s)(u(s)-v(s))ds \Big\|
\nonumber \\ && \le \int_0^{\infty}|k(s)|\|u(t-s)-v(t-s)\|ds
\nonumber \\ && \le \Big(\int_0^{\infty} |k(s)|ds\Big) \|u-v\|_{\infty}
\nonumber \\ && \le \|k\|_{L^1(\mathbb{R}^+)} \|u-v\|_{\infty}.
\end{eqnarray}
Using the above estimate, we get
\begin{eqnarray}
\|Fu(t)-Fv(t)\| && \le \int_{-\infty}^t \frac{CM}{1+|\omega|s^{\alpha}}ds L_f(1+\|k\|_{L^1(\mathbb{R})})\|u-v\|_{\infty}
\nonumber \\ && \le \frac{CM|\omega|^{-\frac{1}{\alpha}}\pi}{\alpha \sin(\frac{\pi}{\alpha})}L_f(1+\|k\|_{L^1(\mathbb{R}^+)})\|u-v\|_{\infty}.
\end{eqnarray}
From our assumption
$$\Lambda=\frac{CM|\omega|^{-\frac{1}{\alpha}}\pi}{\alpha \sin(\frac{\pi}{\alpha})}L_f(1+\|k\|_{L^1(\mathbb{R}^+)}) <1$$ and we get
 $$\|Fu-Fv\|_{\infty} < \Lambda \|u-v\|_{\infty}.$$ Thus the mapping $F$ is a contraction and hence by using Banach contraction principle, we conclude that there exists an unique asymptotically almost automorphic solution of equation (\ref{EE2}).
\end{proof}
 In the next result, we drop the Lipschitz condition. And it is no wonder if we need few more hypothesis, as in like every situation if you relax one thing then you have to pay for the other. So let us assume that,
 \begin{itemize}
 \item[(H4)] The function $f(t,u,\phi)$ is uniformly continuous on any bounded subset $\Omega \subset X\times X$ uniformly in $t\in \mathbb{R}$ and every bounded subset $\Omega \in X \times X,$ the set $\{f(\cdot,u,\phi): (u,\phi) \in X\times X\}$ is bounded in $AAA(\mathbb{R}\times X, X).$
 \item[(H5)] There exists a continuous nondecreasing function $W:\mathbb{R} \rightarrow \mathbb{R}$
such that for each $(t, u, \phi) \in \mathbb{R}\times X\times X, \ \|f(t,u,\phi)\| \le W(\|u\|+\|\phi\|).$
\end{itemize}

 \begin{theorem}
 Assume that $f \in AAA(\mathbb{R} \times X, X)$ satisfying the assumptions $(H_1)-(H_3)$ and the following conditions,
 \item (i) For each $r>0,$
 $$\lim_{t\rightarrow \infty} \frac{1}{h(t)} \int_{-\infty}^t \frac{W(rh(s))}{1+|\omega|(t-s)^{\alpha}}ds=0,$$ where $h$ is a function given in Lemma \ref{lemma3}. Denote
 $$\beta(r)=CM\Big\|\int_{-\infty}^t \frac{W(rh(s))}{1+|\omega|(t-s)^{\alpha}}ds \Big\|_{h}.$$
 \item (ii) For each $\epsilon >0,$ there exists $\delta>0$ such that for every $u,v \in \mathcal{C}_h(X), \|u-v\|_{h} <\delta$ implies that
  $$\int_{-\infty}^t \frac{\|f(s, u(s), Ku(s))-f(s, v(s), Kv(s))\|}{1+|\omega|(t-s)^{\alpha}}ds \le \epsilon$$ for all $t \in \mathbb{R}.$
  \item (iii) For each $a, d \in \mathbb{R}$ and $r>0,$ the set $$\{f(s,h(s)x, K(h(s)x): a \le s \le d, x \in \mathcal{C}_h(X), \|x\|_h \le r\}$$ is relatively compact in $X.$
  \item (iv) $\liminf_{\xi \rightarrow \infty}\frac{\xi}{\beta(\xi)}>1.$ \\
      Then the equation admits an unique mild solution in $AAA(\mathbb{R} \times X, X).$
 \end{theorem}

       \begin{proof}
       Define the operator by
       $$(\Xi u)(t)=\int_{-\infty}^t S_{\alpha}(t-s)(f(s,u(s),Ku(s))ds,$$ for $t \in \mathbb{R}.$

One can observe the followings,
\item (I) For $u \in \mathcal{C}_h(X),$ we get $\frac{\|\Xi u(t)\|}{h(t)} \le CM  \int_{-\infty}^t \frac{W(\|u\|_hh(s))}{1+|\omega|(t-s)^{\alpha}}ds$ and hence $\Xi$ is well defined.
\item (II) For any $\epsilon >0,$ there exists $\delta>0$ satisfying condition $(ii)$ such that for $u, v \in \mathcal{C}_h(X), \|u-v\|_h \le \delta,$ we get
    $$\|(\Xi u)(t)-(\Xi v)(t)\| \le CM \int_{-\infty}^t \frac{\|f(s, u(s), Ku(s))-f(s, v(s), Kv(s))\|}{1+|\omega|(t-s)^{\alpha}}ds \le \epsilon,$$ which implies that $\Xi$ is continuous.
   \item (III) Now we shall show that $\Xi$ is completely continuous. Denote $B_r(X)=\{ x \in X : \|x\| \le r\}.$ Let $V=\Xi(B_r(\mathcal{C}_h(X)))$ and $v=\Xi(u)$
for $u \in B_r(\mathcal{C}_h(X)).$
By the continuity of $S_{\alpha}(\cdot)$ and the condition $(iii)$ of $f,$ we conclude that $$D=\{S_{\alpha}(s)f(\tau, h(\tau)x): 0 \le s \le t, 0 \le \tau \le t, \|x\|_h \le r\}$$ is relatively compact. Also we get $V_b(t) v_b(t) \in t\overline{c_0(D)},$ where $\overline{c_0(D)}$ is the convex hull of $D,$ which shows that $V_b(t)$ is relatively compact for each $t \in [0,b].$ Here subscript $b$ is for the $t \in [0,b].$
Our next aim is to show that $V_b$ is equicontinuous.
\begin{eqnarray}
&&v(t+s)-v(t) \nonumber \\ &&=\int_{-\infty}^{t+s}S_{\alpha}(t+s-\tau)f(\tau, u(\tau), Ku(\tau))d\tau \nonumber \\ &&-
\int_{-\infty}^{t}S_{\alpha}(t-\tau)f(\tau, u(\tau), Ku(\tau))d\tau \nonumber \\ &&=
\int_{-\infty}^{t}S_{\alpha}(t-\tau)(f(\tau+s, u(\tau+s), Ku(\tau+s))-f(\tau, u(\tau), Ku(\tau)))d\tau.
\end{eqnarray}
Using the assumption $(H5)$ we get that $V_b$ is equicontinuous.
From the condition $(i),$
we have
$$\frac{\|v(t)\|}{h(t)} \le \frac{CM}{h(t)} \int_{-\infty}^t \frac{W(rh(s))}{1+|\omega|(t-s)^{\alpha}}ds \rightarrow 0$$ as $t \rightarrow \infty.$ because the above relation is independent of $u \in B_r(\mathcal{C}_h(X))$ and thus $V$ is relatively compact set in $\mathcal{C}_h(X).$
\item (IV) Let us denote $u^{\lambda}(\cdot)$ a solution of equation $u^{\lambda}=\lambda \Xi(u^{\lambda})$ for some $\lambda \in (0,1).$ Now using the estimate
    $$\|u^{\lambda}\| \le CM \int_{-\infty}^t \frac{W(\|u^{\lambda}\|_h h(s))}{1+|\omega|(t-s)^{\alpha}}ds \le \beta(\|u^{\lambda}\|_h)h(t),$$ we get
    $\frac{\|u^{\lambda}\|_h}{\beta(\|u^{\lambda}\|_h)} \le 1.$
    Using the condition $(iv),$ we have
    $\{u^{\lambda}: u^{\lambda}=\lambda \Xi(u^{\lambda})\}, \lambda \in (0,1)$ is bounded.
    \item (V) Further from the assumptions $t \rightarrow f(t,u(t)) \in AAA(\mathbb{R}\times X, X),$ when $u \in AAA(\mathbb{R}, X).$ We can also see that $\Xi(AAA(\mathbb{R}\times X, X)) \subset AAA(\mathbb{R}\times X,X).$ Also $AAA(\mathbb{R}\times X, X)$ is a closed subspace of $\mathcal{C}_h(X).$ Consider the map $\Pi : \overline{AAA(\mathbb{R}\times X, X)} \rightarrow \overline{AAA(\mathbb{R}\times X, X)}.$
        This map is completely continuous. By using the well known Leray Schauder alternative theorem, we deduce that $\Pi$ has a fixed point $u \in AAA(\mathbb{R}, X),$ which is asymptotically almost automorphic mild solution of our fractional integro differential equation.
       \end{proof}
\begin{remark}
Using the transformation $t-s=s_1,$ the integral in conditions $(i)-(ii)$ could replaced by
$$\lim_{t\rightarrow \infty} \frac{1}{h(t)} \int_0^{\infty} \frac{W(rh(t-s))}{1+|\omega|s^{\alpha}}ds=\lim_{t\rightarrow \infty} \frac{1}{h(t)} \int_{-\infty}^t \frac{W(rh(s))}{1+|\omega|(t-s)^{\alpha}}ds=0.$$
and
 $$\int_0^{\infty} \frac{\|f(t-s, u(t-s), Ku(t-s))-f(t-s, v(t-s), Kv(t-s))\|}{1+|\omega|s^{\alpha}}ds \le \epsilon.$$
\end{remark}

\begin{remark}
       If we assume the Holder continuity of $f,$ that is
       $$\|f(t,u, Ku)-f(t,v, Kv)\| \le (L+\|k\|_{L^1{\mathbb{R}^+}}) \|u-v\|^{\theta},$$ where $\theta \in (0,1)$ for all $t \in \mathbb{R}$ and $u \in X.$ Note that $Ku=0$ when $u=0.$ Let us assume the following
       \begin{itemize}
\item[(1)] $f(t,0,0)=p(t).$
\item[(2)] $\sup_{t \in \mathbb{R}} \int_{-\infty}^t \frac{h^{\theta}(s)}{1+|\omega|(t-s)^{\alpha}}ds=\frac{\gamma}{CM} <\infty.$
\item[(3)] For each $a, d \in \mathbb{R}$ and $r>0,$ the set $$\{f(s,h(s)x): a\le s \le d, x \in \mathcal{C}_h(X), \|x\|_h \le r\}$$ is relatively compact in $X.$
\item[(4)] $\liminf_{\xi \rightarrow \infty} \frac{\xi}{\beta(\xi)}>1.$
\end{itemize}
Then the fractional integro differential equation has at least one asymptotically almost automorphic mild solution.

It is not difficult to see the above assertion is true just by doing some setting. Let us denote $\gamma_0=\|p\|, \gamma_1=L+\|k\|_{L^1(\mathbb{R}^+})$
and take $W(\xi)=\gamma_0+\gamma_1{\xi}^{\theta}$ and hence the condition $(H5)$ is satisfied.
 From the condition $(2)$ the function $f$ satisfies condition $(i)$ of previous theorem. Also for each $\epsilon>0$ there exists a $\delta$ satisfying
 $0 < \delta^{\theta} < \frac{\epsilon}{\gamma_1 \gamma_2}$ such that for each $u, v \in \mathcal{C}_h(X), \|u-v\|_h \le \delta$ implies that
 $$CM\int_{-\infty}^t \frac{\|f(s, u(s), Ku(s))-f(s, v(s), Kv(s))\|}{1+|\omega|(t-s)^{\alpha}}ds \le \epsilon.$$ From the definition, assumption $(III)$ easily follows for $W$ and hence we infer that the equation (\ref{EE2})
posses at least one asymptotically almost automorphic mild solution. Hence we conclude the result.
\end{remark}
\section{Examples}
\textbf{Example-1:} Let us consider the following modified fractional relaxation oscillation equation initially defined by Agarwal et.al. \cite{agarwal},
\begin{eqnarray}
\partial_t^{\alpha}w(t,x)&&=\partial^2_{x}w(t,x)-\mu w(t,x) \nonumber \\ &&+\partial_t^{\alpha-1}\Big(\beta w(t,x)(\cos t+\cos \sqrt{2}t)+\beta e^{-|t|}\sin(w(t,x)) \nonumber \\ &&+\sin(\int_{-\infty}^t e^{t-s}w(t,s)ds) \Big), \quad t \in \mathbb{R}, \ x \in [0, \pi], \nonumber \\
&& w(t,0) =w(t,\pi)=0, t \in \mathbb{R},  \nonumber \\
&& w(t,\xi)=w_0(\xi),  \quad \xi \in [0, \xi], \ \mu>0,
 \end{eqnarray}
 where $w_0 \in L^2[0,\pi].$ Define the linear operator $A$ on $X=(L^2([0,\pi]),\|\cdot\|_2)$ by $Aw=w^{''}-\mu w$ with the domain
 $$D(A)=\{w \in X: w^{''} \in X, w(0)=w(\pi)=0\}.$$
 It is well known that $\Delta w =w^{''}$ is the infinitesimal generator of analytic semigroup on $L^2[0,\pi]$ \cite{pazy} and thus $A$ is sectorial of type $\omega=-\mu <0.$
 Denote $w(t)x=w(t,x)$ and $$f(t,w, Kw)(x)=\beta w(x)(\cos t+\cos \sqrt{2}t)+\beta e^{-|t|}\sin(w(x))+\sin(Kw(x))$$ for each $w \in X.$ One can easily see that the function $f(t,u, Ku)$ is asymptotically almost automorphic in $t$ for each $u \in X.$ Now under the condition
 $$|\beta|+1<\frac{\alpha \sin{\frac{\pi}{\alpha}}}{3CM|\mu|^{-\frac{1}{\alpha}}\pi},$$ there exists an unique asymptotically almost automorphic mild solution.
\\\\
\textbf{Example-2:} One can also consider the following fractional order delay
relaxation oscillation equation for $\alpha \in (1,2),$
              \begin{eqnarray}
\frac{\partial^{\alpha}u(t,x)}{\partial
t^{\alpha}}&=&\frac{\partial^2u(t,x)}{\partial
x^2}-pu(t,x)+\frac{\partial^{\alpha-1}}{\partial
t^{\alpha-1}}(f(t,u(t,x),u(t-\tau,x))), \ \tau>0, \nonumber \\
&& \quad \quad t \in \mathbb{R}, \ x \in (0,\pi) \nonumber \\
 u(t,0)& = & u(t, \pi) = 0, \quad t \in \mathbb{R}, \nonumber \\
u(t,x)&=& \phi(t,x) \quad t \in [-\tau,0],
             \end{eqnarray}
where $p>0$ and $f$ is a weighted pseudo almost automorphic or
Weyl almost automorphic function in $t.$ Also assume that $f$
satisfies Lipschitz condition in both variable with Lipschitz
constant $L_f.$
Note that $\int_{-\infty}^tk(t-s)u(s)ds=\int_{-\infty}^tk(-s)u_t(s)ds=J(u_t),$ which can be thought like function of $u_t$ and
hence can be considered as functional differential equations.
Using the transformation $u(t)x=u(t,x)$ and define
$Au=\frac{\partial^2 u}{\partial x^2}-pu, \ u \in D(A),$ where
$$D(A)=\Big\{u \in L^2((0,\pi),\mathbb{R}), u^{'} \in
L^2((0,\pi),\mathbb{R}), u^{''} \in
L^2((0,\pi),\mathbb{R}),u(0)=u(\pi)=0 \Big\},$$ the above equation
can be transform into
              \begin{eqnarray}
\frac{d^{\alpha}u(t)}{dt^{\alpha}}=Au(t)+\frac{d^{\alpha-1}}{dt^{\alpha-1}}g(t,u(t),u_t(-\tau)),
              \end{eqnarray}
$t \in \mathbb{R}$ and $u(t)=\phi(t) \ t \in [-\tau,0].$ It is to
note that $A$ generates an analytic semigroup $\{T(t) :t\ge 0$ on
$X,$ where $X=L^2((0,\pi),\mathbb{R}).$ Hence $pI-A$ is sectorial
of type $\omega=-p<0.$ Further $A$ has discrete spectrum with
eigenvalues of the form $-k^2; k \in N,$ and corresponding
normalized eigenfunctions given by
$z_k(x)=(\frac{2}{\pi})^{\frac{1}{2}}\sin(kx).$ As $A$ is
analytic. Let us assume that
$$\frac{2L_f|\omega|^{\frac{-1}{\alpha}}\pi}{\alpha \sin
\frac{\pi}{\alpha}}<1.$$ Thus under all the required assumption on
$f,$ the existence of weighted almost automorphic solutions is
ensured accordingly. Similarly for the existence of Weyl almost
automorphic solutions, we assume that $2L_f\|S_{\alpha}\|<1.$
\\\\
\textbf{Discussion:} After H. Bohr introduced the concept of almost periodicity \cite{bohr} and then Bochner introduced the concept of almost automorphy, there have been many important generalization of these functions like:
\begin{itemize}
\item[i.] Pseudo almost periodic and automorphic,
\item[ii.] Asymptotic almost periodic and automorphic,
\item[iii.] Weighted pseudo almost periodic and automorphic
\item[iv.] Stepanov almost periodic and automorphic,
\item[v.] Stepanov type pseudo almost periodic and automorphic,
\item[vi] Stepanov type weighted pseudo almost periodic and automorphic,
\end{itemize}
and many more. For the origin and details of these function, one may see \cite{diagana1,ngur} and reference therein. The the application of these function in the area of differential equations attracted many mathematicians and extensive research have been done. In recent year, the application of these function in the field of fractional differential equations got lot of attention after introduction of the resolvent operator $S_{\alpha}(t),$ for details, we refer to \cite{araya, baj}. In this work, we have shown the existence of asymptotically almost automorphic solution of a fractional integro differential equation. The main tool we used are fixed point theorems.

\end{document}